\documentclass[12pt]{amsart} 
\usepackage{amsmath,amstext,amsfonts,amssymb} 

\newtheorem{theorem}{\bf Theorem}[section] 
 
\newtheorem{prop}[theorem]{\bf Proposition} 
\newtheorem{corol}[theorem]{\bf Corollary}

\newtheorem{lemma}[theorem]{\bf Lemma} 
\newtheorem{appl}[theorem]{\bf Application}

\large

\def\cc#1{\hfill\ $#1$\ \hfill} 
\def\tvi{\vrule height 13pt depth 4 pt width 0.2pt} 
\def\tvj{\vrule height 15pt depth 7 pt width 0.2pt}

\newcommand{\abs}[1]{\lvert#1\rvert}

\newcommand{\nd}{{\text{ and }}}

\newcommand{\sm}{{\smallsetminus}}

\newcommand{\lb}{{\lambda}} 
 
\newcommand{\va}{{\varepsilon}}

\newcommand{\A}{{\mathbb A}}

\newcommand{\cA}{{\mathcal A}} 
\newcommand{\cB}{{\mathcal B}} 
 
\newcommand{\cF}{{\mathcal F}} 
 
\newcommand{\cP}{{\mathcal P}} 

\newcommand{\cW}{{\mathcal W}}

\begin{document}

\title{On lattices  of maximal index two} 
\author{Anne-Marie Berg\'e } 
\keywords{Euclidean lattices, perfection
\newline 
  Universit\'e de Bordeaux, UMR 5251, Bordeaux, F-33000, France }
\begin{abstract} 
The maximal index of a Euclidean lattice $L$ of dimension $n$ 
is the maximal index of the sublattices of $L$ spanned by $n$
independent minimal vectors of $L$. In this paper, we prove that a
perfect lattice of maximal index two which is not provided by a cross-section
has dimension at most $5$.
\end{abstract} 
\maketitle

%%Section 1 
\section{Introduction}\label{secintro}

Korkine and Zolotareff proved that an $n$-dimensional
lattice  containing at least 
$\frac{n(n+1)}2$ pairs $\pm x$ of minimal vectors, and spanned by
any subset of $n$ independent minimal vectors, is similar 
to the root lattice $\A_n$.

Here we consider in an  $n$-dimensional  Euclidean space $E$
well rounded lattices, i.e. lattices $L$ the minimal vectors
of   which span $E$. To such a lattice $L$, Martinet
attached some invariants related to the sublattices $M$ of $L$ generated 
by $n$ independant minimal vectors of $L$, in particular 
the set  of possible indices $[L:M]$, and for a given sublattice $M$, the
group structure of the quotient $L/M$.

The  \emph{maximal  index }  of $L$ is :
$$\max_{M}[L:M]\,,$$
where $M$ runs through  sublattices of $L$ spanned by $n$ 
 independent minimal vectors of $L$. 
(Korkine-Zolotareff's result deals with lattices of maximal index $1$.)

In this paper, we consider lattices with maximal index $2$. For such
lattices, the notion of length introduced in [M] can be defined
as follows:

The \emph {length  $\ell\le n$  of a lattice $L$ of maximal index $2$}
is the minimal cardinality
$|X|$  of a set $X$ of
independent  minimal vectors of $L$
 such that $ \sum_{x\in X}x\equiv 0 \mod 2L$.

Up to dimension $7$, there are six perfect lattices with maximal 
index $2$:  in Conway-Sloane's notation (see [C-S] p. 56),
$P_4^1$ and $P_5^1$ have length $\ell=4$, while
$P_5^2$, $P_6^5$, $P_6^6$ and $P_7^{32}$ have length $\ell=5$.
In dimension~$8$, a computation by Batut and Martinet based on the
classification result by Dutour-Sch\"urmann-Vallentin (see [D-S-V])
showed  that no $8$-dimensional perfect lattice has maximal index $2$.

In [M], Martinet conjectured that \emph{ a perfect lattice of maximal
  index~$2$, generated by its minimal vectors, has dimension at most
  $7$.}

In the present work, we prove this conjecture in the case $\ell=n$.
\vskip.1cm 
\begin{theorem}\label{noperf}
A lattice of dimension $n\ge 6$,
of maximal index $2$ and length $\ell=n$, has less than $\frac{n(n+1)}2$
  pairs $\pm x$ of minimal vectors, and in particular is not perfect.
\end{theorem}

Actually, we shall obtain in \ref{theokissing} an
 asymptotic bound 
$$s\le \frac{2n^2}9$$
for the number $s$ of
pairs of minimal vectors much smaller than the (lower) perfection bound
$\frac{ n^2} 2$. 

\section{Notation}

Let $L$ be a lattice of dimension $n\ge 6$, maximal index $2$ 
and  length~$n$. Let $S=S(L)$ and $s(L)=\frac{\abs{S(L)}}2$ denote
the set and number of pairs $\pm x$ of minimal vectors of $L$.

Let $L_0\subset L$ be a sublattice of index $2$ generated
by $n$ independent minimal vectors $e_1,\dots e_n$ of $L$.
We  have $L=\langle L_0,e\rangle$, where, by possibly reducing $e$
modulo $L_0$, and using the definition of the length, we may prescribe
$$e=\frac{e_1+\cdots +e_n}2\,.$$

The hypotheses on the maximal index and the length of $L$ imply
that the minimal vectors of $L_0$  are 
just the $\pm e_i$, and that 
the other possible minimal vectors of $L$ 
are of the form
$$\frac{\pm e_1\pm e_2 \pm\cdots\pm e_n}2\,.$$
 (See [M], Proposition 2.1.) 
In order to prove 
Theorem \ref{noperf}, we may and shall assume that $s(L)\ge n+1$,
and in particular,
by negating some $e_i$ if necessary, \emph {we shall
suppose $e$ itself minimal} (unless otherwise specified in Section
$3$). The next sections are devoted to the other minimal vectors
$x\in S(L)\sm S(L_0)$,
that we represent 
by their  set $I$ of minus signs:
$$x=x_I=e-\sum_{i\in   I}e_i\,.$$
\medskip

We call \emph{ type of $x$} the number $|I|$   of minus signs 
in the expression of $x$ ($e$ is of type $0$).  
Of course the types of $x$ and $-x$ add to $n$, therefore
by possibly negating $x$ we shall suppose 
$|I|\le \frac n 2$, and if $|I|=\frac n2$ we
shall prescribe $1\in I$.
[{\small The index set $I$
associated to the minimal vector $x$, and {\sl a fortiori}  its type,  
depend on the 
choice of $e\in L\sm L_0$.}]
\medskip

 The following notation is relative to a given set
of  $r\ge 3$ minimal vectors $x_1, x_2,\dots, x_r$ in
$L\setminus L_0$ identified to their index sets $I_1,I_2,\dots I_r$
$$x_k=x_{I_k}=e-\sum_{i\in I_k}e_i, \quad I_k\subsetneqq \{1,\dots ,n\},
\quad |I_k|\le \frac n2\,.$$
We denote by  
$$m=|\cup_{k}I_k|\quad (m\le n)$$ the number of indices
involved in the expression of the $x_k$. Actually, we may and shall suppose
that $$\bigcup_{k}I_k=\{1,2,\dots,m\}\,.$$
For 
$i=1,\dots,n$ we call \emph {weight} of $i$ the number
$w(i)=0,\dots,r$ of 
subsets $I_k$ it belongs to; we thus have
$$\sum_{k=1}^rx_k=r\,e-\sum_{i=1}^n w(i)e_i\,.\eqno(1)$$
We also introduce the partition of  $\cup I_k=\{1,\dots,m\}$ 
into sets of indices of   given weights
$$W_k=\{ i\in\cup_{k} I_k \mid w(i)=k \}\quad (1\le k\le r)\,,$$ 
that we regroup  into  the sets of indices of even and 
odd weights
$$\mathcal W_0=W_2\cup W_4\cup\dots \ \nd \ \mathcal W_1=W_1\cup W_3\dots \,.$$

\medskip

Section $3$ gives properties  about the weights in families of $3$, 
$4$ or $5$ minimal vectors; these results are  used in Sections $4$ to $7$ 
to give an upper
bound for the \emph{ number $t_p$ of   minimal vectors of a given
type $p$}.

[{\small The bounds for $t_1$, $t_2$ and $t_1+t_2$ given in Sections~$3$ and
$4$ were obtained by Martinet and the author while giving a
classification of the  six-dimensional perfect lattices based on their
maximal index, work previously done by Baranovskii and Ryshkov in [B-R].}]

Section $8$ concludes by an estimation of the
``kissing number'' $s(L)=n+t_0+t_1+\cdots +t_{\lfloor \frac n
  2\rfloor}$ of $L$ ($t_0=1$) strictly smaller than the dimension
$\frac{n(n+1)}2$ of the space of lattices.

\section{Properties of a set of minimal vectors}
\subsection{Minimal vectors of type 1}
The following property derives from the hypothesis
``\emph {no $n$ independent vectors of $L$ span a sublattice of index $3$ of
$L$}'' and does not suppose $e$ minimal.
\begin{prop}\label{proptype1} %%{\rm(Berg\'e- Martinet)}
Suppose $n\ge 5$. Then
there exist at most four minimal vectors 
of the form $e-e_i$ (i.e. $t_1\le 4$).
\end{prop}

\begin{proof}
Let $x_i=e-e_i$, $i=1,\dots,5$ 
be five minimal vectors of type $1$ of $L$; using (1) we obtain
$$\sum_{i=1}^5 x_i-\sum_{i=6}^ne_i=5e-\sum_{i=1}^n e_i=3e\,;$$
Clearly the $n$ vectors $x_1,\dots,x_5,e_6,\dots,e_n$ are linearly
independent, and generate
a sublattice $L'$ of index $3$ in $L$, a contradiction.
\end{proof}

\medskip

\subsection{Weights in a set of minimal vectors}

These  properties of a set of $r=3,4$ or $5$ minimal vectors
of the form $x_k=e-\sum_{i\in I_k} e_i$ make essential use of
the assumption that $\ell=n$,
i.e. that \emph{any set $X\subset S(L)$ of independent minimal 
vectors satisfying a congruence $\sum_{x\in X}x\equiv 0\ \mod 2L$
has cardinality $|X|=n$}.
We first focus on the case $r=4$, and here again $e$ is not supposed
to be minimal.
\begin{lemma}\label{lemr=4} If every set   
$I_1,I_2,I_3$ and $I_4$ contains at least 
one index of weight $1$, then this index is unique, and there is no
index of weight~$3$.
\end{lemma}

\noindent
{\sl Proof }. 
From (1)  follows
$$\begin{aligned}
\sum_{k=1}^rx_k+\sum_{i\in \mathcal
  W_0}e_i&=re-\sum_{i\in\cW_1}w(i)e_i-\sum_{i\in\cW_0}(w(i)-1)e_i\\
        &=4e-\sum_{i\in W_1\cup W_2}e_i-3\sum_{i\in W_3\cup W_4}e_i
%%-5\sum_{i\in W_5\cup W_6}e_i-\dots
\\
        &=4e-\sum_{i=1}^me_i-2\sum_{i\in W_3\cup W_4}e_i
%%-4\sum_{i\in W_5\cup W_6}e_i-\dots
\,,
\end{aligned}$$
where $\sum_{i=1}^me_i=2e-\sum_{i=m+1}^n e_i$, and thus we obtain
$$
\sum_{k=1}^4x_k+\sum_{i\in \mathcal W_0}e_i-\sum_{i=m+1}^n e_i=%%(r-2)e
2e-2x\,,$$
with $x=\sum_{i\in W_3\cup W_4}e_i\in  L$. Thus  the set
$$X=\{x_1,x_2,x_3,x_4\} \cup \{e_i, i\in\mathcal W_0\text{ or } i\ge
m+1\}$$
of minimal vectors of $L$ (which does not include the vector $e$)
satisfies the  congruence
$$\sum_{k=1}^4x_k+\sum_{i\in \mathcal W_0}e_i+\sum_{i=m+1}^n
e_i \equiv 0\mod 2L\,.$$ 
Its cardinality is 
$$|X|=4+|\mathcal W_0|+(n-m)=n-(|W_1|-4)-|W_3|\,$$
where $|W_1|\ge 4$ since for $k=1,\dots,4$, $W_1\cap I_k\ne\emptyset$.
To complete the proof of the lemma, it remains to prove that 
$X$  is free. Suppose 
$$\sum_{k=1}^4\lb_k x_k+\sum_{i\in \mathcal W_0\cup
 \{m+1,\dots,n\}}\mu_i e_i=0\eqno(2)$$ 
where the $\lb_k,\mu_i$ are real numbers.
Fix $k\in\{1,\dots,4\}$; by assumption, there exists $i_k\in I_k$ of weight $1$,
hence belonging to no other $I_{h}$. With respect to 
the basis $e_1,\dots,e_n$ for
$E$ the coefficient $a_{i_k}$ of the left hand side of (2) 
on the corresponding $e_{i_k}$ reads 
$a_{i_k}=\frac {\sum_h \lb_h}2-\lb_k$. Its vanishing implies that the
$\lb_k$ have a common value $\lb$ 
satisfying $2\lb=\lb$, hence $\lb=0$. Now (2) reduces to
$\sum_{i\in \mathcal W_0\cup  \{m+1,\dots,n\}}\mu_i e_i=0$, and
 all $\mu_i$ are zero.
The set $X$  is free, which completes the proof.
\qed

\medskip

 \emph{From now on, we suppose that $e$  is a minimal vector  of $L$.}

\begin{prop}\label{lemW1} 
If  every $I_k$, $1\le k\le r$,  contains at least one index
of weight $1$, and if moreover when $r=3$ there is an index of
weight~$3$, then $r$ is equal to $3$ or $4$, the index of weight~$1$ 
in every $I_k$ is uniquely determined, and 
  for  $r=3 $ (resp. $4$) we have  $|W_3|=1$ (resp. $|W_3|=|W_4|=0$).
\end{prop} 

\begin{proof} 
The case $r\ge 5$ follows from the case $r=4$ and Proposition
\ref{proptype1}.

(a) {\sl Case $r=3$}. By assumption, there exists 
an index of weight $3$, say $1\in I_1\cap I_2\cap I_3$.
We change $e_1$ into  $e'_1=-e_1$ and $e$ into 
$e'=e-e_1=\frac{e'_1+e_2+\cdots+e_n}2$ (not necessarily minimal), 
and we consider the four
minimal vectors $x_0=e$, $x_1$, $x_2$ and $x_3$
which, relatively to $e'$, read
$x_k=e'-\sum_{i\in I'_k}e_i$ with $I'_0=\{1\}$, and 
$I'_k=I_k\setminus \{1\}$ for $k=1,2,3$. 
The weights $w(i)$ and 
$w'(i)$ of an index $i$ relative to the sets $(I_1,I_2,I_3)$ 
and $(I'_0,I'_1,I'_3,I'_4)$
coincide, except for $i=1$: $w(1)=3$ and $w'(1)=1$. Thus the four
minimal vectors $x_i,\, 0\le i\le 3$ satisfy the hypotheses
of Lemma \ref{lemr=4}: there is no 
index $i\ge 2$ of weight~$3$, and the indices of weight~$1$ in
$I_1,I_2,I_3$ are uniquely determined, as announced.

(b) {\sl Case $r=4$}. It remains to prove that $W_4=\emptyset$. 
Otherwise, any subset of three $I_k$ should satisfy the hypotheses
of (a), hence $|W_4|=1$, 
and by considering convenient ones we should obtain
$|W_2|=0$ (if $W_2\cap I_1\cap I_2\ne \emptyset$,
we consider the subset $\{I_2,I_3,I_4\}$, where $I_2$ has too many
indices of weight one). Since  by the lemma  we already know that
$|W_1|=4$ and $|W_3|=0$,  every $I_k$ should
contain just one index of weight $1$, say $i_k$, and one index of
weight $4$, say $1$: the $x_k$ are of the form $x_k=e-e_1-e_{i_k}$,
where the $i_k\ge 2$ are pairwise distinct.
 By the same substitution $e_1\mapsto e'_1=-e_1$,  
$e\mapsto e'=e-e_1$, ($e'$ is not necessarily minimal), we
obtain five vectors of type $1$, namely  $x_0=e'-e'_1$ and the four
$x_k=e'-e_{i_k}$,  a contradiction with Proposition \ref{proptype1}.

\end{proof}
\begin{appl}\label{excompo}

$\bullet$
Four pairwise disjoint sets $I_k$   are singletons.

$\bullet$
Let $x_0=e-\sum_{i\in I_0}e_i$ be a minimal vector of type 
$p=\abs{I_0}\ge 3$, and 
let $A\subset I_0$, with $1\le \abs{A}\le p-1$. There exists at most
one vector $x_I$ of type $p$ such that $I\cap I_0=A$.
\end{appl} 

\medskip

We now  interchange 
the parts of  even and odd weights, and focus on weight $2$.
\begin{prop}\label{lemW2} 
Let $x_1,\dots, x_r$ be $r\ge 3$ minimal vectors, of the form
$x_k=e-\sum_{i\in I_k}e_i$.
For $1\le k<k'\le r$ we define the relation 
$$I_k\sim I_{k'} \iff  I_k\cap I_{k'}\cap W_2\ne \emptyset\,.$$
We suppose that the graph of the relation $\sim$ is  a cycle of
length $r=3$ or $5$, or a star of valency $3$ (with $r=4$).

 Then 
 the dimension  $n$ is equal to $m$ or $m+1$, where 
$m=|\bigcup_{k=1}^r I_k|$;
moreover,  if $n=m+1$, then $\abs{W_2}=r$ (resp. $r-1=3$) in the case
of a cycle (resp. star).
\end{prop} 
\begin{proof}

Note that the cycle (resp. star) contains $r$ (resp. $r-1$) edges,
and thus that the number $|W_2|$  of indices  of weight $2$ is $\ge r$
(resp. $r-1$).  Since there is nothing specific to prove in the case
$n=m$, we shall suppose $n\ge m+1$ and show that then all equalities 
about $n$ and $|W_2|$ hold.

We consider the set 
$$X=\{x_1,\dots,x_r, e_i\ (i\in \cW_1), \rho e\}$$
of minimal vectors, where  $\rho\in\{0,1\}$ is the remainder of $r$
modulo $2$, i.e. $\rho=1$ in the case of the cycle, and $0$ in the
case of the star.  
Using (1) we obtain that the vectors of $X$  add to a congruence
 modulo $2L$:
$$\sum_{k=1}^rx_k+\sum_{i\in \mathcal W_1}(w(i)-2)e_i+(4-r)e=
\sum_{i\in\cW_0}(2-w(i))e_i+2\sum_{i=m+1}^n e_i\,.
\eqno(4)$$

We now prove that the assumption $n\ge m+1$ implies that $X$ is free. 
 Let
$\lb_k\  (k=1,\dots,r),\ \mu_i\  (i\in \cW_1), \mu $ (equal to zero in
the case of the star) be real numbers such that
$$\sum_{k=1}^r\lb_k x_k+\mu e+\sum_{i\in \mathcal W_1}\mu_ie_i=0\,.\eqno (5)$$ 
Put  $a=\frac{\sum_{k=1}^r\lb_k+\mu}2$. 
With respect to the basis 
$(e_i)$ Condition (5) reads:
$$\left\{
\begin{aligned}
a-\sum_{k\mid i\in I_k}\lb_k&=0 \quad  \forall i\in\cW_0\\
a-\sum_{k\mid i\in I_k}\lb_k +\mu_i&=0\quad   \forall i\in\cW_1\\
a&=0\quad  \forall i\ge m+1. \end{aligned}\right.\eqno(5')
$$
Since $n\ge m+1$, we can write  Condition (5') for $i=n$, and we
obtain  $a=0$, i.e. $\sum \lb_k=-\mu$. 

Now, if $I_k\sim I_{k'}$ are adjacent,
we obtain $\lb_k=-\lb_{k'}$ by writing Condition (5') 
 for some $i \in  W_2\cap I_k\cap I_{k'}$. 
In the case of the $3$-star with node $I_1$, it 
follows $\lb_2=\lb_3=\lb_4=-\lb_1$, 
with $\sum \lb_k=0$ since $\mu=0$, and thus $\lb_k=0$ for all $k$.
In the case of the odd cycle say $(I_1,I_2,\dots,I_r)$, 
the $\lb_k$ takes the values $\lb_1$ and $\lb_2=-\lb_1$
alternatively; since $r$ is odd, all $\lb_k$ vanish again, and so do
$\sum\lb_k$ and $\mu$.

Eventually, in both cases (star or cycle), the conditions (5') give
$\mu_i=0$ for all $i\in \mathcal W_1$.
Thus, when $n\ge m+1$, the set $X$ is free. Since its
vectors add to a congruence modulo $2L$, we must have $|X|=n$, where
$$\begin{aligned}
n-|X|&=(|W_2|-r)+(n-m-1) \text{ in the case of the odd cycle}\\
     &(|W_2|-(r-1))+ |W_4|+(n-m-1) \text{ in the case of the star}\,.
\end{aligned}
$$
The  terms between brackets in the right-hand sides are non-negative, 
and since $n-|X|=0$ they vanish, as stated. 
\end{proof}

%%In the four next sections, we consider subsets $S_p$ of minimal
%%vectors of a given type $p$ (i.e. $x=e-\sum _{i\in I}$, with
%%$\abs{I}=p$), and we give upper bounds for $t_p=\abs{S_p}$.

\section{ Sets of minimal vectors of  type  at most two}

The type $1$  was dealt with in Proposition \ref{proptype1}.
We now focus on the type $2$, i.e. on minimal vectors of the 
form $x=e-e_i-e_j$, $1\le i<j\le n$. 

\begin{theorem} \label{theotype2} %%{\rm (Berg\'e-Martinet)} 
We define on the set $\{1,2,\dots,n\}$ the relation
$$i\equiv j \quad \text{ if and only if } 
\quad e-e_i-e_j \quad \text{ is a minimal vector}\,.$$
Then, if  $n\ge 6$, the graph of the relation $\equiv$
is a subgraph of the complete bipartite graph with $9$ edges,
except for $n=6$ where it can also be a cycle of length $5$.
\end{theorem}
\begin{proof}
We discard isolated vertices. By \ref{lemW1} we
know that the valencies of the vertices are at most equal to $3$, 
and that a disconnected graph  contains no vertex of valency~$3$.
%%and   that if there are at most
%%$4$  minimal vectors, the graph contains at most two (non-trivial) connected
%%components, and by Proposition \ref{lemW1} that a disconnected graph 
%%contains no point of valency $3$. 
By Proposition \ref{lemW2}, the graph of
the relation $\equiv$ contains no triangle (since $n>4$) and no
pentagon except for $n=6$. If the graph is connected
(resp. disconnected), it contains no path of length $\ge 6$
(resp. $\ge 4$) and no cycle of length $\ge 7$ (resp. $5$): otherwise, 
we could extract four minimal vectors whose graph has $3$ connected
components, a contradiction with \ref{lemW1}. 
Now we conclude that every  possible linear graph has at
most $6$ (non-isolated vertices) say $V=\{1,2,3,4,5,6\}$, and that
except the pentagon, they are bipartite: we can define a partition
$V=V_0\cup V_1$, $|V_0|=|V_1|=3$ such that no two vertices in
the same $V_k$ are adjacent.
It remains to consider a connected graph with at least one vertex of
valency $3$, say $1$. Denote by $V_0=\{2,4, 6\}$ the set  of its adjacent
vertices. By Proposition \ref{lemW1}, any other edge  must be
connected to this star, 
i.e. one of its end-points  belongs to $V_0$, but not the other one
(triangles are forbidden). Let $V_1$  denote the set 
vertices adjacent to vertices in $V_0$. It contains at most $3$
 vertices, as we shall now prove. If a vertex in $V_0$, say $2$, 
has valency  $3$, exchanging the roles of the vertices $1$ and $2$, 
we see that  $V_1$ is  the set of the  vertices adjacent to
$2$, and thus $|V_1|=3$. If no vertex in $V_0$
 has valency $3$, distinct vertices in $V_1\setminus \{1\}$ are
 adjacent to distinct vertices in $V_0$. Suppose that there are three
 of them, say $3$, $5$, $7$, respectively adjacent to $2$, $4$, $6$.
We have then four edges, namely $1\ 2$, $3\ 2$, $5\ 4$ and $7\ 6$  in
three connected components, a contradiction. Thus $|V_1|=3$.
\end{proof}

\begin{corol}  \label{cort12} %%{\rm (Berg\'e-Martinet)} 
Let $t_1$ and $t_2$ be the number of minimal vectors of types $1$ and
$2$ respectively. Then $t_1+t_2\le 9$, where  equality holds
only if the  graph of the relation $\equiv$ is the complete bipartite graph
($(t_1,t_2)=(0,9)$) or if it consists of two non-adjacent nodes of
valency $3$ and their common adjacent vertices ($(t_1,t_2)=(3,6)$).
\end{corol}
\begin{proof}
Since by Proposition \ref{proptype1} we know that $t_1\le 4$,
we only need to consider the graphs (in the sense
of the theorem) with $t_2\ge 5$ edges. We first note that if $i$ is an
isolated point of the graph, $e-e_i$ cannot be minimal: we
could extract from the $t_2\ge 5$ edges of the graph three disjoint
ones, or two secant and a third one disjoint, which, together with
{i}, would contradict Proposition \ref{lemW1}. We now consider the
case of a pentagon say $(1,2,3,4,5,1)$. If there are $4$ vectors of
type $1$, three of them correspond to consecutive vertices, say
$e-e_1$, $e-e_2$ and $e-e_3$, which together with $e-e_4-e_5$,
contradicts again \ref{lemW1}. Thus $t_1\le 3$ and $t_1+t_2\le 8$.
The other  graphs to consider are
included in the complete bipartite graph associated with, say,  the
partition $\{2,4,6\}\cup \{1,3,5\}$.
We first consider a path of length $5$, say $1-2-3-4-5-6$, and suppose
$e-e_i$ minimal ($i=1,\dots,6$). Then $i=1$ is not possible, because
the four sets $I_1=\{1\}$,
$I_2=\{2,3\}$,  $I_3=\{4,5\}$, $I_4=\{5,6\}$  
 contradict  Proposition \ref{lemW1}. The same argument, with
 $I_2=\{1,2\}$ instead of $\{2,3\}$, forbids $i=3$. So the only
possible values of $i$ are $i=2$ and $i=5$, and $t_1\le 2$,
$t_1+t_2\le 7$.
Now consider the cycle $(1,2,3,4,5,6,1)$. By considering the 
path $1-2-3-4-5-6$, we see that $e-e_1$ is not minimal, and
since all vertices play the same role, we conclude that $t_1=0$.
This conclusion extends to any subgraph of the complete bipartite 
graph containing such a cycle, i.e. the complete graph itself, and the
ones obtained by suppressing one edge,  two disjoint edges or three
pairwise disjoint edges.

One more graph with $5$ edges contains no node of valency $3$: the
disjoint union of a cycle of length $4$, say  $(1,2,3,4)$,
and a path of length $1$. Suppose $e-e_1$ minimal; 
the four sets $I_1=\{1\}$, $I_2=\{2,3\}$,  $I_3=\{3,4\}$, $I_4=\{5,6\}$ 
contradict \ref{lemW1}. Thus there are at most two minimal vectors of
type $1$, namely $e-e_5$ and $e-e_6$, and $t_1+t_2\le 7$.

We are left with graphs  which contain at least 
one node of valency $3$, say $1$, with adjacent 
vertices $2,4,6$. If $e-e_i$ is minimal, we must have
$i\in\{1,2,4,6\}$ (otherwise, the four sets of indices
$\{1,2\},\ \{1,4\},\ \{1,6\}$ and $\{i\}$ should 
contains indices of weight one--respectively $2$, $4$, $6$ and $i$--
but also an index of weight $3$, which contradicts 
Proposition \ref{lemW1}). Note that the four values $i=1,2,4,6$
are never simultaneously possible: since there are more than four
edges, one of the vertices $2,4,6$, say $2$, has 
another adjacent vertix, say $3$. Then if $e-e_1$, $e-e_4$, $e-e_6$
were minimal vectors, the sets $\{1\}$,  $\{4\}$,  $\{6\}$  and
$\{2,3\}$ would contradict Proposition \ref{lemW1}. We then have 
$t_1\le 3$, which completes the case $t_2=5$. 

If there are, in the graph we consider,
two adjacent nodes of  valency $3$, say $1$ and $2$, the only possible 
minimal vectors of type $1$ are thus $e-e_1$ and $e-e_2$. 
In particular, the graph of $t_2=7$ edges obtained by suppression from
the complete bipartite graph two secant edges, say $3-2$ and $3-4$
contains three nodes of valency $3$, namely $1$, $5$ and $6$, where 
$6$ is adjacent to $1$ and $5$. The unique minimal vector of type $1$
is thus $e-e_6$, and $t_1+t_2\le 8$. This completes the case $t_2=7$.

We are left with $3$ non-isomorphic graphs with $6$ edges.
If it is obtained by suppressing (from the complete graph)
 the three edges of a path of length $3$, it contains two adjacent
 nodes of valency $3$, and $t_1\le 2$, as announced.
The same conclusion is valid for the graph obtained by suppression of 
three edges, two of them secant, for instance $4-5$, $5-6$ and $2-3$.
The resulting graph contains the disjoint union of the cycle $(14361)$
with the edge $2-5$, and we have seen that the only possible minimal
vectors are $e-e_2$ and $e-e_5$, and again $t_1\le 2$.
But for the graph obtained by suppressing  three secant edges, say
$5-2$, $5-4$ and $5-6$, it contains two non-adjacent nodes, and it is
consistent with $t_1=3$ minimal vectors of type $1$, namely $e-e_2$,
$e-e_4$, $e-e_6$. The proof of the corollary is now complete.
\end{proof}

\section{ Configurations of three vectors of type $p\ge 3$}

The graph we consider is that of the relation
$\sim$ introduced in Proposition \ref{lemW2}:
$I_k\sim I_{k'}$ if $I_k\cap I_{k'}$ contains an index of weight $2$.

\begin{prop}\label{proptrio} Let
$x_1=e-\sum_{i\in I_1}e_i,\ x_2=e-\sum_{i\in I_2}e_i,\
x_3=e-\sum_{i\in I_3}e_i$ be three vectors
of the same type $p\ge 3$. We suppose that $W_3=\cap I_k$ is not empty.
Then, 
if $(p,n)\ne (4,8)$, one at least of the sets $I_1$, $I_2$ and $I_3$
has no index of weight one, and the 
 $\sim $-graph is a path.
\end{prop}

\begin{proof}
For all $k=1,2,3$ we put $a_k=|W_1\cap I_k|$, $b_k=W_2\cap_{h\ne k}
I_h$, and $c=|W_3|$. We have $p=|I_k|=a_k+(|W_2|-b_k)+c$,
and thus the $a_k-b_k$ have a common value $p-|W_2|-c$.

\medskip
$\bullet$ Suppose first  $a_k\ge 1$ for all $k$.
Then by Proposition \ref{lemW1}, we have $a_1=a_2=a_3=c=1$.
Thus, the $b_k$ have a common value $|W_2|/3$, where 
$2|W_2|=3p-|W_1|-3|W_3|=3p-6$; $p$ is even, $|W_2|=\frac{3(p-2)}2$,
and  the $b_k=\frac{p-2}2$ are  non-zero. We conclude that
the $\sim$-graph is a cycle, and by Proposition \ref{lemW2} 
we must have $n\le m+1$, where $m=\sum|W_i|=\frac32 p+1$.  
The unique solution for the inequalities $2p\le n\le \frac32 p+2$
is $p=4$, $n=8$.

$\bullet$ Now we suppose $a_1=0$, and thus $p=b_2+b_3+c$. Since
$|I_1\cap I_2|=b_3+c$ is $<p$, $b_2$ is non-zero, and so is $b_3$.
From $|W_2|=b_1+p-c$ follow $3p-m=|W_2|+2|W_3|=b_1+p+c$ 
and  $m=2p-c-b_1$. If the graph were a cycle, i.e. if $b_1\ge 1$, we
should obtain $m\le 2p-2$ (since $c=|W_3|\ge 1$), and thus $m+1<2p$, a
contradiction with Proposition \ref{lemW2}.
We conclude that the graph is the path $I_2\sim I_1\sim I_3$.
\end{proof}
\begin{corol}\label{nocycle}
We suppose $(p,n)\ne (4,8)$, and we consider
three distinct minimal vectors $x_0=x_{I_0}, x=x_I$
and $x'=x_{I'}$  of the same type $p\ge 3$ such that $I$ 
and $I'$ both intersect $I_0$.  
We put 
$$I=A\cup X \text{  with }\quad A=I\cap I_0,\quad X=I\smallsetminus A \,,$$
$$I'=A'\cup X' \text{  with }\quad A'=I'\cap I_0,\quad X'=I\smallsetminus A \,.$$
Then:

(i)  The $\sim$-graph 
$(I,I',I_0)$ is a cycle if and only if $A\cap A'=\emptyset$ and 
 $X\cap X'\ne\emptyset$.

(ii) If $A$ and $A'$ satisfy an  inclusion,  so do $X$ and $X'$.

(iii) If $A$ and $A'$ satisfy no  inclusion and if $(I,I',I_0)$ is not
a cycle, then $X$ and $X'$ are disjoint; if moreover
 $A\cap A'\ne \emptyset$, then   $I_0=A\cup A'$. 
\end{corol}
{\sl Proof.}
The sets of indices of weight one in $I_0$, $I$ and $I'$ are
respectively $I_0\setminus (A\cup A')$, $X\setminus X'$ and
$X'\setminus X$;
the sets of indices of weight two in $I_0\cap I$, $I_0\cap I'$ and
$I\cap I'$ are respectively   $A\setminus A'$, $A'\setminus A$ and 
$X\cap X'$; eventually the set $W_3$  of indices of weight three is $A\cap
A'$.

(i) We see directly that $(I,I',I_0)$ is a cycle if and only if $A$
and $A'$ satisfy no inclusion and $X$ and $X'$ intersect. From
Hypothesis $(p,n)\ne (4,8)$, Proposition \ref{proptrio} shows that 
it can only happen when $A\cap A'=\emptyset$.

(ii) Assume $A'\subset A$. Then $W_3=A'$ is not empty. From
Proposition \ref{proptrio},  one at least of the sets $I$, $I'$ and
$I_0$ contains no index of weight one. Since $I_0\setminus (A\cup
A')=I_0\setminus A$ is not empty (since $I\ne I_0$ have the same
cardinality), $X$ and $X'$ satisfy an inclusion, namely 
 $X\subset X'$ (since $|A|+|X|=|A'|+|X'|=p$). 
In particular, if $A=A'$, we have  $X\subset X'$ and $|X|=|X'|$,
thus  $X=X'$ and $I=I'$.

(iii)  Now we assume that $A$ and $A'$ satisfy no inclusion, i.e.
 $I\sim I_0$ and $I'\sim I_0$. As already noted in (i), we must have
$I\nsim I'$, i.e. $X\cap X'=\emptyset$.
If moreover $A\cap A'\ne \emptyset$, by Proposition \ref{proptrio} one 
at least of the sets $W_1\cap I=X$, $W_1\cap I'=X'$ and 
$W_1\cap I_0=I_0\setminus (A\cup A')$ is empty, thus $I_0=A\cup A'$.
\qed

\medskip

The end of the section is devoted to the special 
case $(p,n)=(4,8)$.
\begin{prop} \label{4,8} If $n=8$, then $t_4\le 6$.
\end{prop}

\begin{proof}
Indeed, since $n=2p$ we may and shall assume that all index sets
contain $1$. Let $I_0=\{1,2,3,4\}$ be one of them;
from \ref{excompo} we know that
there is at most one $I\ne I_0$ for a given $I\cap I_0$.
Let $I$ be such that $|I\cap I_0|=3$.  For instance, 
put $I_1=\{1,2,3,5\}$,
and let $I_2=\{1,2,4,a\}$, with $a\in \{5,6,7,8\}$. Actually, if
$a=5$, the configuration $\{I_0, I_1, I_2\}$ is a cycle  in the sense of
\ref{lemW2}, which is absurd since $n=8>5+1$. Thus for instance
$I_2=\{1,2,4,6\}$ and similarly $I_3=\{1,3,4,7\}$. Now, let  $x_I$ be a
minimal vector such that $|I\cap I_0|=1$, i.e. 
$I=\{1,a,b,c\}$, with $\{a,b,c\}\subset \{5,6,7,8\}$. Actually,
by Proposition \ref{lemW1} we must have $\{a,b,c\} =\{5,6,7\}$.
Otherwise, if for instance $5\notin \{a,b,c\}$, 
the configuration $\{I_0,I_1,I\}$  would contain 
too many  indices of weight 1. Now we have  $I\sim I_k$ for 
$k=1,2,3$ and the configuration $\{I_1, I_2,I_3, I\}$ 
contradicts Proposition \ref{lemW2} ($n=8,m=7,|W_2|=6>3$).
We conclude that there are at most $3$ minimal vectors $x_I$ such that 
$|I\cap I_0|=1$ or $3$. Now, we shall prove that  there are at most $2$
minimal vectors $x_I$ such that $\abs{I\cap I_0}=2$.
Otherwise,  let $I=\{1,2\}\cup X$, $I'=\{1,3\}\cup X'$ and 
$I"=\{1,4\}\cup X"$ be three solutions ($X$, $X'$ and $X"$ subsets of
$\{5,6,7,8\}$). In their configuration $(I,I',I")$, $2$, $3$ and $4$
have weight $1$, and by \ref{proptrio}  the elements of $X$,
$X'$ and $X"$ must have weight $\ne 1,3$, i.e.
they have weight $2$, wich leads (up to permutation) to
$I=\{1,2,5, 6\}$, $I'=\{1,3,6,7\}$ and $I"=\{1,4,5,7\}$. Now  the
configuration $I_0, I,I',I"$ satisfies the hypotheses of Proposition
\ref{lemW2}, with $m=7$ and $n=8$, but $\abs{W_2}=6$, a contradiction.
\end{proof}

\medskip

Taking into account Proposition \ref{4,8}, we discard in the next
sections the case $(p,n)=(4,8)$.

\section{Families  without cycles of length $3$}
 
\begin{theorem} \label{theoI1I2} Let $\{x_I, I\in \cF\}$ be a set 
of minimal vectors of the same type $p\ge 3$,
such that $\cF$ contains no cycles of length $3$.

If  $p\ge 4$ and $n\ge 2p+2$, then  $|\cF|\le p+6$.

If $p\ge 4$ and $n= 2p+1$, or $p=3$ and $n\ge 8$, then   $|\cF|\le p+5$.

If $(p,n)=(3,7)$, then $\abs{\cF}\le 7$.

If $p\ge 3$, $p\ne 4$ and  $n=2p$, then $\abs{\cF}\le p+1$.
 \end{theorem}

[{\small Note that the bound $p+6$  is indeed reached, as checked
for $(p,n)=(4,10), (5,12),\dots $}.]

The whole section is devoted to the proof of this theorem. We first
consider the case when
all elements of $\cF$ intersect a given one, which includes the case 
$p=\frac n2$, since then we may prescribe that all $I$ contain a given
index.  

\begin{prop}\label{theoI0}
We suppose that 
there exists $I_0\in\cF$ such that for all $I\in \cF$, $I\cap I_0\ne
\emptyset$. 

If $ p\ge 4$ (resp. $p=3$) and $n\ge 2p+1$, then $\abs{\cF}\le p+3$
(resp. $5$);

if $ p\ge 3$ and $n= 2p$, then $\abs{\cF}\le p+1$.
\end{prop}
\begin{proof}

For $I\in\cF$, we write  
$$I=A\cup X, \text{\ with \ } A=I\cap I_0\ne\emptyset 
\text{\ and \ } X=I\smallsetminus A\,.$$
 From \ref{excompo} it follows that  $I$ 
is uniquely specified by $A$ (or equivalently by $X$).
We shall now describe the set
$$\cF_0=\{I\cap I_0,\, I\in\cF \}$$
in one-to-one correspondence with $\cF$, and prove that
it consists of one or two totally ordered sequences, except in the
following case.
\begin{lemma}\label{lem3I0} 
Let $I=A\cup X$, $I'=A'\cup X'$ and $I"=A"\cup X"$
be three elements of $\cF\smallsetminus\{I_0\}$ such that $A$, $A'$ and
$A"$ satisfy no pairwise inclusions.
Then $|\cF|=4$.
\end{lemma}
{\sl Proof of the lemma.} From Corollary \ref{nocycle} we see
that the sets $X$, $X'$ and $X''$ are pairwise disjoint. 
First, we prove that $A$, $ A'$ and $A''$  are pairwise disjoint.
Otherwise suppose for instance $A\cap A'\ne \emptyset$, and thus,  by 
  Corollary \ref{nocycle}, $I_0=A\cup A'$, i.e. $I_0\smallsetminus A\subset
  A'$;  since $A"$ is not included in  $A'$, it is no more included in
$I_0\smallsetminus A$ i.e. $A"\cap A\ne \emptyset$; and of course we also
have $A"\cap A'\ne \emptyset$. In the configuration $\{I,I',I"\}$,
the indices in  $X$, $X'$ and $X'$  have weight one, thus by
Proposition \ref{proptrio} $I\cap I'\cap I"=\emptyset$, so that
the indices in 
$A\cap A'$, $A'\cap A"$ and $A\cap A"$ have weight $2$, and 
 $\{I,I',I"\}$ is a cycle, a contradiction: $A$, $A'$ and $A"$ are
 pairwise disjoint as announced. Now in the configuration 
 $\{I,I',I", I_0\}$, the sets of weight one in $I$, $I'$, $I"$ and
 $I_0$  are respectively equal to $X\ne\emptyset$, $X'\ne\emptyset$,
 $X"\ne\emptyset $  and 
$I_0\smallsetminus(A\cup A'\cup A")$. If this last set were not empty, we
should obtain from Proposition \ref{lemW1}
$|X|=|X'|=|X"|=|I_0\smallsetminus(A\cup A'\cup A")|=1$, thus 
$|A\cup A'\cup A"|=3(p-1)=p-1$, i.e. $p=1$, absurd.
So $I_0=A\cup A'\cup A"$ is a partition of $I_0$.
Let $J=B\cup Y$ be another element of $\cF$. Then $B=J\cap I_0$
satisfies an inclusion with at most one of the sets $A$, $A'$, $A"$
(since these sets are mutually disjoint, and so are their complements
$X$, $X'$ and $X"$). Thus $B$ satisfies no inclusion with at least 
two among $A$, $A'$, $A"$, say $A$ and $A'$. Thus $I_0=A\cup A'\cup B$ 
is a partition of $I_0$, and $B=A"$, a contradiction.
Thus  $\cF=\{I_0, I,I',I"\}$.
\qed

\medskip

If $\cF_0=\{I\cap I_0,\, i\in\cF\}$   is a totally ordered family, we have
$\abs{\cF_0}=\abs{\cF}\le \abs{I_0}=p$, and Proposition \ref{theoI0}
is proved. The same conclusion holds in 
the situation of Lemma \ref{lem3I0}. 
 We therefore consider in $\cF$
$$I=A\cup X \text{ \,and\, }  J=B\cup Y \text{ \,  such that \,}
 A\not\subset B \text{ \, and \,} A\not\supset B\,,$$
and may suppose that 
for any  $K=C\cup Z\in \cF$,  $C\in\cF_0$ satisfies an inclusion
with $A$ or $B$, or equivalently (by \ref{nocycle}),
 $Z$ satisfies an inclusion with
$X$ or $Y$, which are disjoint. If $Z$ satisfies an inclusion with
both $X$ and $Y$, i.e. if $Z\supset X\cup Y$, then
by Proposition \ref{proptrio} applied to $\{I,J,K\}$ (since the
indices of $C\ne \emptyset$ have  weight~$3$, and those of
 $A\sm B\ne \emptyset$ and $B\sm A\ne \emptyset$
have weight~$1$), $Z$ coincides with $X\cup Y$, and $C$ with $A\cap B$.

Now  choose a  pair $(A,B)$ with  $\abs{A}$ and $\abs{B}$ 
minimal: if $C\ne A\cap B$ satisfies an inclusion with $A$
(resp. $B$), it contains $A$ (resp. $B$). 
[{\small If $C\varsubsetneq A$, from the 
minimality of the pair $(A,B)$, $C$ satisfies an inclusion with $B$
too, and we just saw that $C=A\cap B$.}]
 Thus, apart from $I_0$ and (possibly) $A\cap B$, $\cF_0$ is union
of two disjoint sets
$$\cA=\{ C\in\cF_0\,\mid\, A\subset C\varsubsetneq I_0\} \text{ \, and \,} 
\cB=\{ C\in\cF_0\,\mid\, B\subset C\varsubsetneq I_0\}\,.$$
{\sl These sets are totally ordered by inclusion}, as we now prove.
Consider for instance in $\cF$ two distinct
elements $K=C\cup Z\in \cF$ and $K'=C'\cup Z'\in \cF$
with $C\supset A$ and $C'\supset A$, i.e. by \ref{nocycle}, $Z$ and
$Z'$ included in $X$. Thus $Z$ and $Z'$ are disjoint from $Y$, and 
by \ref{nocycle} again, $B$ satisfies no inclusion with $C$ or $C'$.
Since $\abs{\cF}\ge 5$, Lemma \ref{lem3I0} implies that $C$ and $C'$
satisfy an inclusion: the set $\cA$ is totally ordered by inclusion,
and so is $\cB$.
We then have 
$$\abs{\cA}\le p-\abs{A}\text{ \, and \,} \abs{\cB}\le p-\abs{B}\,.$$
We have to  consider two cases:

{\sl case $A\cap B\ne\emptyset$, in particular $n=2p$}.
It follows from Corollary \ref{nocycle} that
$I_0=A\cup B$ and then $|A|+|B|=p+|A\cap B|\ge p+1$.
It implies $\abs{\cA}+\abs {\cB}\le 2p-\abs{A}-abs{B}\le p-1$,
and taking into account $I_0$ and $A\cap B$, $\cF=\cF_0\le p+1$ as
required.

{\sl case $A\cap B=\emptyset$}. 
We then have $\cF_0=\{I_0\} \cup \cA\cup \cB$,
where $\cA$ and $\cB$ are totally ordered sequences
$A= A_1\subsetneq A_2\subsetneq \dots \subsetneq A_k\subsetneq I_0 $
and $B= B_1\subsetneq B_2\subsetneq \dots \subsetneq B_h\subsetneq I_0 $ 
(every pair $(A_i,B_j)$ without inclusion),with for instance 
$1\le k\le h\le p-1$.
We then have $\abs{\cF}\le 1+h+k$. If $k=1$, we obtain $\abs{\cF}\le p+1$,
and Proposition \ref{theoI0} is proved in this case.
The same conclusion holds if $h\le 2$, for instance if $p=3$, since
then  we obtain $\abs{ \cF }\le 5$. We thus suppose $h\ge 3$ and
$k\ge 2$. and consider  the four elements
of $\cF$ corresponding to $A$, $B$, $A_3$ and $B_2$, say
 $I=A\cup X$, $J=B\cup Y$, $I_3=A_3\cup X_3$ and $J_2=B_2\cup Y_2$,
with  $X$ and $Y$ disjoint and $X_3\subsetneq X$ and $Y_2\subsetneq Y$.
Their respective
subsets of weight~$1$ are  $X\sm X_3\ne\emptyset$, 
$Y\sm Y_2\ne\emptyset$, $A_3\sm (A\cup B_2)$ and $B_2\sm
(A_3\cup B$). If $A_3\cap B_2$ were empty,
we should have $\abs{ A_3\sm (A\cup B_2)}=\abs {A_3\sm A}\ge 2$ and 
$\abs{ B_2\sm (A_3\cup B}=\abs{ B_2\sm B}\ge 1$, a contradiction with
Proposition \ref{lemW1}. Thus $A_3\cap B_2$ is not  empty, and from
\ref{nocycle}, it follows that $\abs{A_3}+\abs{B_2}\ge p+1$.
Now from  $h=\abs{\cA}\le 2+(p-A_3)$ and  $k\le 1+(p-B_2)$ we obtain
$\cF\le 4+2p-(\abs{A_3}+\abs{B_2})\le p+3$ as required.
\end{proof}

 \medskip

We know that (for $p>1$), the family $\cF$ contains at most
three pairwise disjoint elements. We examine now this case.
\begin{prop}\label{theo123}
If $\cF$ contains  three pairwise disjoint elements, then $\cF\le p+3$.
%%then it  contains at most $p+3$ elements, except for 
%%$(p,n)=(3,9)$ or $(3,10)$, where it may contain $6$ elements.
\end{prop}
\begin{proof}
This  will follow from the more precise result \ref{prop123}, for which
we need some more notation.

Let $I_1, I_2,I_3$ be three elements of $\cF$ pairwise disjoint.
For every $I\in \cF$ distinct from the $I_j$ we consider the partition
$I=A_1\cup A_2\cup A_3\cup X$,  where  $ A_j=I\cap I_j$.
%%where some of the $A_j$ may be empty (not all, from \ref{excompo}).
Actually, $X$ is empty.
Otherwise, we could apply  Proposition \ref{lemW2} to the subset
$\{I_1,I_2,I_3,I\}\in\cF^4$, whose sets of indices of weight one
$I_j\smallsetminus A_j, j=1,2,3$ and $X$ should have just one element,
and $I$ should have  $p=3(p-1)+1$  elements, i.e. $p=1$, a
contradiction.
We thus have 
$$I=A_1\cup A_2\cup A_3,\qquad  \text{  where  } A_j=I\cap I_j.$$
We introduce the following subsets of $\cF$:

For $i=1,2,3$, $\cF_i$ is the set of $I\in\cF$ with only $A_i$
non-empty. We have just proved that  $\cF_i=\{I_i\}$.

For $1\le i<j\le 3$, $\cF_{ij}$ is the set of $I\in\cF$ 
with only $A_i$ and $A_j$ non-empty. 

Eventually, $\cF_{123}$ is the set of $I\in\cF$ intersecting $I_1$,
$I_2$ and $I_3$. %% We note that if $I$ belongs to $\cF_{123}$,
%% the $\sim$-graph of the family $I,I_1,I_2,I_3$ is a star using 
%%$3p$ indices,  and by Proposition \ref{lemW2}, $n=3p$ or $3p+1$;
%%actually, $n=3p+1$ implies $|A_i|=1$ for all $i$ and thus $p=3$.

\begin{lemma}\label{prop123}
The four  subsets $\cF_{ij}$ and $\cF_{123}$ are empty but one. We have
$\abs{ \cF_{ij}}\le p-1$  and  $\abs{\cF_{123}}\le 3$.
\end{lemma}
{\sl Proof of \ref{prop123}.}\quad We may suppose that $\abs{\cF}\ge 5$.
Let  $I\ne I'$ be two elements of $\cF$ distinct from $I_1,I_2,I_3$.
There exists $i\in \{1,2,3\}$  
such that both $A_i$  and $A'_i$ are non-empty, for instance 
we  suppose $A_1$ and $A'_1$ non-empty, and we are in the situation
described by  Corollary \ref{nocycle} with $I_1$ in the r\^ole of $I_0$.
We have to consider two cases.

{\sl Case 1.} $A'_1\subsetneq A_1$. Then by \ref{nocycle} we have
$A_2\cup A_3\subset A'_2\cup A'_3$, i.e. 
 $A_2\subset A'_2$  and
 $A_3\subset A'_3$. 
As $I$ is distinct from $I_1$, $A_2$ for instance is non-empty, and so
is $A'_2$. Then by \ref{nocycle}, the inclusion $A_2\subset A'_2$
implies now $A'_3\subset A_3$, and thus $A_3= A'_3$. 
Since $I$ and $I'$ are distinct, from \ref{excompo}  we conclude that
$A_3$ and $A'_3$ are empty, i.e. that $I$ and $I'$ both lie in
$\cF_{12}$.
 
{\sl Case 2. $A_1$ and $ A'_1$ satisfy no inclusion.}
Then by  \ref{nocycle} 
$A_2\cup A_3$ and $A'_2\cup A'_3$ are disjoint, i.e 
 $A_2\cap A'_2=A_3\cap A'_3=\emptyset$. 
Then 
$I$ and $I'$ do not belong to distinct $\cF_{ij}$,
as we now prove.
If $(I,I')$ lies in $\cF_{12} \times \cF_{13}$, Proposition \ref{lemW1},
applied to the four elements  $I$, $I'$, $I_2$, $I_3$, gives 
$|I_2\smallsetminus A_2|=|I_3\smallsetminus A'_3|=1$, thus $A_1$ and
$A'_1$ are disjoint singletons, and the same proposition applied to 
 $\{I,I',I_1,I_3\}$ gives, since $p\ge 3$  ( $I_1\supsetneq A_1\cup
 A'_1$) $|A_2|=1$ and thus $|I|=p=1+1$, a contradiction.

We  may assume for instance that $A_2$ and $A'_2$ are both non-empty. Then
permuting $I_1$ and $I_2$ we conclude that $A_1$ and $A'_1$ are
disjoint two. We then have, for $i=1,2,3$, $|A_i|+|A'_i|\le p$.
Since $\sum_i(|A_i|+|A'_i|)=|I|+I'|=2p$, we have two possibilities. 

1)  $|A_1|+|A'_1|=|A_2|+|A'_2|= p$, and thus $A_3=A'_3=\emptyset$: 
$I$ and $I'$ lie in  $\cF_{12}$. Note that in this case,
$\cF_{12}$ reduces to the pair ${I,I'}$, since $I'=(I_1\sm A_1)\cup
(I_2\sm A_2)$ is uniquely determined by $I$.

2) Otherwise, there are at least two sums  $|A_i|+|A'_i|< p$, say 
 for $i=2$ and $i=3$. By use of Proposition \ref{lemW1}
applied to the set $\{I,I',I_2,I_3\}$, we obtain (since $A_1\cap
A'_1=\emptyset$) $|A_1|=|A'_1|=|I_2\smallsetminus (A_2\cup A'_2)|=
|I_3\smallsetminus (A_3\cup A'_3)|=1$; since the third sum
$|A_1|+|A'_1|=2$ is also $<p$, we obtain $|A_i|+|A'_i|=p-1$ for all
$i$. We then have $p=3$, and for all $i$, $A_i$ and $A'_i$ are
disjoint singletons. Thus $I$ and $I'$ both 
lie in $\cF_{123}$. [{\small Note that Proposition \ref{lemW2} applied
  to the star $I,I_1,I_2,I_3$ shows that in this case, $n=9$ or $10$}]

\medskip

We conclude that two distinct elements $I$ and $I'$ of $\cF$ belong
to the same subset $\cF_{ij}$ or $\cF_{123}$, therefore only one of
them is non-empty. Moreover, the elements of $\cF_{123}$ are pairwise
disjoint, thus by \ref{excompo} there are at most three of them.
Eventually, if $\cF_{ij}$ is not empty, it consists either of
a disjoint pair $(I, (I_i\cup I_j)\sm I)$, or of at most $p-1$
elements  $I=A_i\cup A_j$, where the set $\{A_i\}$ is totally ordered
by inclusion. 
This completes the proof of \ref{prop123} and thus of 
Proposition \ref{theo123}.

\bigskip

{\sl We now come back to the proof of Theorem \ref{theoI1I2}}.
Taking into account the result of \ref{theoI0} and \ref{theo123}, we may
 and will assume now  that $\cF$ contains two disjoint elements, say $I_1$
 and $I_2$, such that every $I\in\cF$ intersects at least one of them.

In other terms,  there is a partition $\cF=\cF_1\cup \cF_2\cup \cF_{1,2}$, with

 $\cF_1=\{I\in\cF\quad\mid\quad I\cap I_1\ne \emptyset\nd I\cap I_2=
 \emptyset\},$

 $\cF_2=\{I\in\cF\quad\mid\quad I\cap I_2\ne \emptyset\nd I\cap I_1=
 \emptyset\},$

$\cF_{1,2}=\{I\in\cF\quad\mid\quad I\cap I_1\ne \emptyset\nd I\cap
I_2\ne \emptyset\}$.

Since by Proposition \ref{theoI0} we know that $|\cF_1\cup \cF_{1,2}|\le
p+3$ ($p+2$ if $p=3$), the proof of Theorem \ref{theoI1I2} 
will result from the following proposition.

\begin{prop} \label{propF1F2}
We have $\min(|\cF_1|,|\cF_2)|)\le 3$, where equality holds only
when $n\ge 2p+2$.
\end{prop}

We keep the notation $I=A_1\cup X$ for an element $I\ne I_1$ in
$\cF_1$, where $A_1=I\cap I_1$ and $X=I_1\smallsetminus A_1$, and similarly
$J=B_2\cup Y$ ($B_2=J\cap I_2$, $Y=I_2\smallsetminus B_2$) for an element
$J\ne I_2$ of $\cF_2$.

\begin{lemma}\label{lemF1F2}
Let $I=A_1\cup X$  be an element of $\cF_1\sm \{I_1\}$.
Then there is \emph{ at most one } $J=B_2\cup Y\in\cF_2$ 
such that $Y$ satisfies no inclusion with $X$, and this may occur
only when $\abs{X}=1$ (and obviously $Y$ also is a singleton).
\end{lemma}
{\sl Proof of the lemma}. Let $J=B_2\cup Y\in\cF_2$ be distinct from
$I_2$ such that $Y\not\subset X$ and $Y\not\supset X$. 
With respect to the set $\{I_1,I_2,I,J\}$, the subsets of indices of
weight one in $I_1,I_2,I,J$ are respectively are $I_1\smallsetminus
A_1$,   $I_2\smallsetminus B_2$,  $X\smallsetminus
Y$ and  $Y\smallsetminus X$, all of them non-empty, and by  
Proposition \ref{lemW1} all of them singletons. From
$|A_1|=|B_2|=p-1$, follows $|X|=|Y|=1$.  
Now, let $J'=B'_2\cup Y'$
 be another solution in $\cF_2$, i.e. with $Y'$ singleton distinct
from $X$, and also from $Y$ since $J'\ne J$. The subsets of indices of
weight one in $I_1,I,J,J'$ respectively are $I_1\smallsetminus A_1\ne
\emptyset$, $X$,  $(B_2\smallsetminus B'_2)\cup Y$,  $(B'_2\smallsetminus
B_2)\cup Y'$, the last two subsets with $p-1>1$ elements, which
contradicts Proposition \ref{lemW1}. The solution $J$ is unique.
\qed

\medskip

{\sl Proof of Proposition \ref{propF1F2}.} 
For it we may and will assume that $\cF_1$ contains at least two elements
$I=A_1\cup X$ and $I'=A'_1\cup X'$ distinct from $I_1$. We fix such a
pair $I,I'$ and suppose for instance $|X|\le |X'|$.
We now prove that there is in $\cF_2$  at most one $J=B_2\cup Y$ with
$Y$ satisfying an inclusion with $X$. 

1) First suppose that $A_1$ and $A'_1$  satisfy an inclusion, for instance 
$A'_1\subsetneqq A_1$. Then with respect to the set $\{I,I',J\}$, $I$
and $J$ have indices of weight one (those of $A_1\smallsetminus A'_1$
and $B_2$ at least). Since $I\cap I'\cap J=X\cap Y$ is not empty,
Proposition \ref{proptrio} implies that $I'$ has no index of weight one, i.e.
$X'\subset X\cup Y$. The inclusion between $X$ and $Y$ is thus
$X\subset Y$, and we conclude that $Y$ contains $X$ and $X'$. 
Now, let 
$J'=B'_2\cup Y'$ be another solution with $Y'$ satisfying 
an inclusion with $X$. Then $Y$ and $Y'$ containing $X'$ should
intersect, thus satisfy an inclusion, and 
 $B'_2$ and $B_2$ too.
Hence we might exchange the r\^ole of the pairs $(I,I')$ and $(J,J')$,
and conclude that
 $X$ and $X'$ must contain $Y$ and  $Y'$, thus $X=X'=Y'=Y$, a contradiction.

2) We now suppose that $A_1$ and $A'_1$ satisfy no inclusion. We may have
two types of solutions $J=B_2\cup Y$, with $Y$ satisfying an inclusion
with $X$.

$\bullet$ \emph{ Type I: $Y$  satisfies an inclusion with $X'$ too}.
Then we  must have $Y\supset X\cup
X'$ since  $X$ and $X'$ are disjoint.  Actually, if $Y$ contains
strictly $X\cup X'$, we may apply Proposition \ref{lemW1} to the set
$\{I,I',J,I_2\}$, with sets of indices of weight one
$A_1\smallsetminus A'_1$, $A'_1\smallsetminus A_1$, $I_2\smallsetminus
B_2$ and $Y\smallsetminus X\cup X'$, and conclude $|B_2|=p-1$, thus
$|Y|=1$, a contradiction. Hence we have $Y=X\cup X'$, which determines
entirely  $J$ in $\cF_2$. 

$\bullet$ \emph{ Type II: $Y$  satisfies no inclusion with $X'$ }.
We know by Lemma
\ref{lemF1F2} that such a solution is unique and implies
$|Y|=|X'|=1$, and thus  $|X|=1$ too ($|X|\le |X'|$).  More precisely
since $Y$ and $X$ satisfy an inclusion, we have $Y=X=\{x\}$ and  
$X'=\{x'\},\, x'\ne x$. 

It remains to prove that we cannot have simultaneously in $\cF_2$ 
solutions of types I and II. We then suppose $X=\{x\}$ and  
$X'=\{x'\},\, x'\ne x$, and we consider in $\cF_2$ an element 
$J=B_2\cup Y$ of the first type, i.e. with $Y=\{x,x'\}$, and 
an element of the second type $J'=B'_2\cup Y'$ with $Y'=\{x\} $. 
Since $Y'\subset Y$, we have $B_2\subset B'_2$. We may thus apply 
the part 1) to $I$, $J$ and $J'$ (since $X$ satisfies an  inclusion 
with $Y$), and conclude that $X$ must contain $Y$ and $Y'$,
a contradiction.

We have then proved that in every case there is at most
one element $J=B_2\cup Y$ in $\cF_2$ such that $Y$ satisfies an
inclusion with $X$, and by Lemma \ref{lemF1F2} we obtain
$\abs{\cF_2}\le 1+1+1$.

 In order to complete the proof of \ref{propF1F2},
it remains to observe that if $\cF_2$ contains (apart from $I_2$) two elements 
$J=B_2\cup Y$ and $J'=B'_2\cup Y'$, then $|Y\cup Y'|\ge 2$ (since
$Y\ne Y'$ by \ref{lemF1F2} and \ref{nocycle}) and therefore
 $\abs{\cup_{I\in \cF}I}\ge \abs {I_1} + \abs {I_2}+2=2p+2$.

\end{proof}
\medskip

\section{Families  with cycles of length $3$}

\begin{theorem}\label{theocycle}
Let $\{x_I,I\in \cF\}$ be a set of minimal vectors of the same type
$\abs{I}=p\ge 3$. We suppose that $\cF$ contains a cycle of length
$3$, and that $(p,n)\ne (4,8)$. 
Then 

(1) the dimension $n$ of the lattice satisfies 
$2p+1\le n\le 3p-2$ ;

(2) we have $\abs{\cF}\le n$, and even, if $n=3p-2$ and $p\ge 4$, 
$\abs{\cF}\le p+2$.
\end{theorem}

The section is devoted to the proof of the theorem. 
Let $(I_1,I_2,I_3)$ be in $\cF$ a fixed cycle for the relation
$\sim$. We use for this cycle the notation and rules of Section 2.
In particular $W_k\subset \cup_{h=1,2,3}I_h$ is 
the set of indices of weight $k$, $k=1,2,3\, $,  
and $m=\abs{I_1\cup I_2\cup I_3}$. 

(1) Since $(p,n)\ne (4,8)$  we know by \ref{proptrio}
 that $W_3=I_1\cap I_2\cap I_3$ is empty,
which allows us to prescribe $n\ge 2p+1$.
We thus  have
$$W_2=E_{12}\cup E_{23}\cup  E_{13} 
\quad \mid\quad  E_{ij}=I_i\cap I_j\ne\emptyset$$
since $I_i\sim I_j$. 
From the relations $m=\sum_k\abs{W_k}$ and
$3p=\sum_kk\abs{W_k}$ we obtain 
$$\abs{W_2}=3p-m\,;$$   
hence  $\abs{W_2}\ge 3$ reads $m\le 3p-3$, and the inequality
$n\le 3p-2$ follows from \ref{lemW2} ($n$ is equal to
$m$ or $m+1$).

For every  permutation $(i,j,k)$ of $\{1,2,3\}$,  let
$$E_k=I_k\sm(E_{ki}\cup E_{kj})$$ 
denote the set of
indices of weight $1$ in $I_k$. 
The condition $\abs{I_k}=p$ reads
 $p=\abs{E_k}+\abs{E_{ki}}+\abs{E_{kj}}$, which
implies $\abs{E_k}\le p-2$, and even $\abs{E_k}= p-2$ if $n=m+1$. 
Moreover it proves that $\abs{E_k}-\abs{E_{ij}}$ does not depend on $k$:
$$\Delta=\abs{E_k}-\abs{E_{ij}}=p-\abs{W_2}=m-2p\ge 0\,,$$
where equality holds when $n=m+1$ and thus 
(since $\abs{W_2}=3$ by \ref{lemW2})  
$(p,n)=(3,7)$.  We thus have $\abs{E_k}\ge 1$, with equality if and
only if $(p,n)=(3,7)$.

\medskip

In the following,
$$I=A_1\cup A_2\cup A_3\cup A_{12}\cup A_{23}\cup A_{13}\cup (I\cap
\{m+1\})\,,$$
$$\text{with }\mid A_i=I\cap E_i,\, A_{ij}=I\cap E_{ij}$$
denotes an element of $\cF$ distinct from $I_1,I_2,I_3$.

The discussion below is based on the $A_{ij}$, starting with the case
when they are empty.

\begin{lemma}\label{lemnoAij}
  If $I$ intersects no $E_{ij}$, 
then $n=m+1$ belongs to $I$ and $(p,n)$ is equal to $(3,7)$  or
$(4,10)$. In the first case there is at most one
such $I$, say $I=E_1\cup E_2\cup \{7\}$;
 in the second case there are at most two of them,
say $I=\{a_1, a_2,a_3,10\}$ and $I'=\{a'_1,a'_2,a'_3,10\}$, where
$E_i=\{a_i,a'_i\}$.
\end{lemma}
\begin{proof}
By assumption $I$ is of the form $I=A_1\cup A_2\cup A_3\cup(I\cap
\{m+1\})$ and  the condition
$\abs{I}=p$ reads $p=\abs{A_1}+\abs{A_2}+\abs{A_3}+\va$, where
$\va=\abs{I\cap \{m+1\}}$. 
From $\abs{A_i}\le \abs {E_i}\le p-2$  it
follows that at most two $A_i$ are non-empty, say $A_1$ and $A_2$. 
Then, $(I,I_1,I_2)$ {\emph is  a cycle}.
Put $m_{12}=\abs{I\cup I_1\cup I_2}$. By Proposition \ref{lemW2}, 
we must have $n=m_{12}$ or $m_{12}+1$, 
and thus $\abs{m-m_{12}}\le 1$, where 
$m_{12}-m=\va -\abs{E_3\smallsetminus A_3}$.

{\sl Case}  $m_{12}=m+1$,  i.e. $n=m+1$, $\va=1$ and $A_3=E_3$.
By \ref{lemW2}, we have $\abs{E_3}=p-2$, and thus 
$p=\abs{A_1}+\abs{A_2}+(p-2)+1\ge p+1$, a contradiction.

{\sl Case}  $m_{12}=m-1$, i.e. 
 $p=\abs{A_1}+\abs{A_2}+\abs{E_3}-1$. But now $n\ge m$ is equal to 
$m_{12}+1$, and by applying \ref{lemW2} to the cycle
$(I,I_1,I_2)$ we obtain $\abs{A_1}=\abs{A_2}=1$ and thus
$\abs{E_3}=p-1$, a contradiction.

{\sl Case}  $m_{12}=m$, i.e. 
 $p=\abs{A_1}+\abs{A_2}+\abs{E_3}$. 

If $\va=0$, $A_3=E_3$ is
 non-empty, and we can interchange $I_2$ and $I_3$; for the cycle
$(I,I_1,I_3)$ we can discard as above the cases $m_{13}=m\pm 1$, and
thus $m_{13}=m$, with again $\va=0$, and thus $A_2=E_2$, and
similarly $A_1=E_1$. We conclude that
$p=\abs{E_1}+\abs{E_2}+\abs{E_3}=
\abs{E_1}+\abs{E_2}+\abs{E_{12}}+\Delta=p+\Delta$
implies $\Delta=0$ and thus $(p,m,n)=(3,6,7)$. 
Then  the graph of  $\{I,I_1,I_2,I_3\}$ is a star of centre $I$ 
with six indices of  weight two, 
 which contradicts  Proposition \ref{lemW2}.

We are left with the case  $\va=1$, $\abs{A_3}=\abs{E_3}-1$.
Since $n=m+1$, we have, by \ref{lemW2}, $n=3p-2$ and $\abs{E_i}=p-2$,
and thus  $\abs {A_3}=p-3$; 
from $n=m_{12}+1$, we obtain $\abs{A_1}=\abs{A_2}=1$.  

If $(p,n)=(3,7)$, then $A_3=\emptyset$
and $I=E_1\cup E_2\cup\{7\}$. Let $I'$ be another
solution of this type, for instance $I'=E_1\cup E_3\cup\{7\}$. Then
we can apply Proposition \ref{proptrio} to the set
$\{I,I',I_1\}$, since $I\cap I'\cap I_1=E_1$ is not empty. But
$I$,$I'$ and $I_1$ have  indices  of weight one (respectively
those of $E_2$, $E_3$ and  $E_{12}$), a contradiction.

If $p\ge 4$, $A_3$ is not empty and we may interchange (as above) 
$I_3$ with $I_1$ or $I_2$, and obtain $\abs{A_3}=1$, which implies
$p=4$, and $n=10$. Let $I=A_1\cup A_2\cup A_3\cup\{10\}$ and 
$I'=A'_1\cup A'_2\cup A'_3\cup\{10\}$ be two distinct solutions, for
instance $A_3\ne A'_3$. If $A_1=A'_1$, consider as 
above the set  $\{I,I',I_1\}$. It has
indices of weight $3$ (those of $A_1$), and also of weight $1$ in $I$ 
($A_3\smallsetminus A'_3=A_3$), in $I'$ ($A'_3$) and in $I_1$ ($E_{12}\cup
E_{13}$), a contradiction with Proposition \ref{proptrio}, since
$(p,n)\ne (4,8)$. We conclude that for all $i$, $A_i$ is distinct from
$A'_i$, i.e. since $E_i$ has two elements, $A_i$ and $A'_i$ are
complementary in $E_i$. 
\end{proof}
\begin{lemma}\label{lemAij=Eij}
Here we suppose that the $A_{ij}=I\cap E_{ij}$ are  not all empty.
Then

(i) $m+1$ does not belong to $I$;

(ii) there exists a pair $(i,j)$ such that $A_{ij}=E_{ij}$, unique except
for $(p,n)=(3,7)$, where $I_0=E_{12}\cup E_{23}\cup E_{13}$ may belong
to $\cF$;

(iii) we have, for $(i,j,k)$ permutation of $\{1,2,3\}$,
$$A_{ij}=E_{ij}\Longleftrightarrow A_k=\emptyset\,;$$ 

(iv) if $A_{12}=E_{12}$ and if $\emptyset\subsetneq A_{13}\subsetneq E_{13}$,
then $(A_1,A_{23})=(E_1,\emptyset)$.
\end{lemma}
\begin{proof} We suppose for instance $A_{12}\ne \emptyset$. 
 By Proposition \ref{proptrio}, we know that
 {\emph  the graph of $\{I,I_1,I_2\}$ is a path} 
(since $I\cap I_1\cap
  I_2=A_{12}\ne\emptyset$), and that
  its vertex of valency $2$ has no index of weight $1$.
The sets of indices of weight one in $I$, $I_1$ and $I_2$ are respectively
$(I\cap\{m+1\}) \cup A_3$, $(E_1\sm A_1)\cup
(E_{13}\sm A_{13} $) and $(E_2\sm A_2)\cup
(E_{23}\sm A_{23} $). The sets of indices of weight two
in $I\cap I_1$, $I\cap I_2$ and $I_1\cap I_2$ are respectively
$A_1\cup A_{13}$, $A_2\cup A_{23}$, and $E_{12}\sm
A_{12}$.

First suppose $A_{12}\ne E_{12}$. Then in the path above
 $I_1$ and $ I_2$ are adjacent, and one of them
 has  valency two and thus contains no index 
of weight one, the other one is not adjacent to $I$. Thus
$$\emptyset\subsetneq A_{12}\subsetneq E_{12}
\Longrightarrow (A_1,A_2,A_{13},A_{23})=
\begin{aligned}& (E_1,\emptyset,E_{13},\emptyset)\text{\quad or\quad}\\
                &(\emptyset,E_2,\emptyset,E_{23})
\end{aligned}\,,$$
which  establishes the ``existence part'' of (ii), and (up to exchange
of $2$ and $3$) the item (iv). 

Suppose for instance $A_{12}=E_{12}\,.$
Then, we have a path $I_1\sim I\sim I_2$ (since $I_1$ and $I_2$ are no
more adjacent), $I$ has no index of weight one: $m+1$ does not belong to $I$ as
stated in (i), and $A_3$ is empty, as stated in the part $\Rightarrow$
of(iii). 

Conversely, suppose $A_3=\emptyset$. If $\emptyset\subsetneq
 A_{12}\subsetneq E_{12}$, we obtain $I=E_1\cup E_{13}\cup
 A_{12}\subset I_1$, thus $I=I_1$, a contradiction. If
 $A_{12}=\emptyset$, we may suppose (by (ii))
for instance $A_{13}=E_{13}$, and by the part $\Rightarrow$ of (iii),
 $A_2=\emptyset$: $I=A_1\cup A_{23}\cup E_{13}$, with
$A_1$ and  $A_{23}\ne\emptyset$ (otherwise, $I$ should be a strict
subset or $I_3$ or $I_1$). Then,  the set $\{I,I_1,I_2\}$ is a cycle,
with $m'=\abs{I\cup I_1\cup I_2}=m-\abs{E_3}\le m-1$. By Proposition
\ref{lemW2}, we must have $m'=m-1$ and $n=m'+1$. By \ref{lemW2} again,
the last equality implies $\abs{I\cap I_1}=1$, i.e. $\abs{A_1\cup
  E_{13}}=1$, a contradiction.
Thus, $A_{12}=E_{12}$, as announced.

It remains to discuss the ``unicity'' in (ii).
Suppose $A_{12}=E_{12}$ and $A_{13}=E_{13}$ for instance. We then have 
$A_2=A_3=\emptyset$, and $I=A_1\cup E_{12}\cup E_{13}\cup A_{23}$,
with $A_{23}\ne\emptyset$ (since $I\not\subset I_1$). By (iv), we
conclude $A_{23}=E_{23}$, since $A_{12}$ and $A_{13}$ are non empty.
Thus all $A_i$ are empty and $I$ coincides
with $E_{12}\cup E_{13}\cup E_{23}$ i.e. the set $W_2$ of indices of
weight $2$ in the cycle $(I_1,I_2,I_3)$. Equaling the cardinalities
we obtain $p=\abs{W_2}=3p-m$, thus $m=2p$ and $(p,n)=(3,7)$. 
\end{proof}

\medskip
Apart from  the two ``exotic'' solutions for
$(p,n)=(3,7)$ or $(4,10)$ exhibited in \ref{lemnoAij} and
\ref{lemAij=Eij}, we just have proved that  the set $\cF\sm \{I_1,I_2,I_3\}$ 
is a disjoint union of three components $\cF_1$, $\cF_2$ and $\cF_3$,
where $\cF_i=\{I\in\cF, I\ne I_1,I_2,I_3 \mid\  I\cap
E_i=\emptyset\}$. We now evaluate their cardinality.

\begin{lemma} \label{lemE12full}
 The set $\cF_3=\{I\in\cF, I\ne I_1,I_2\mid (A_3,A_{12})=(\emptyset,E_{12})\}$
contains at most $p-\abs{E_{12}}-1$ elements.
\end{lemma}

\begin{proof}
Let 
$$I=A_1\cup A_2\cup E_{12}\cup  A_{13}\cup  A_{23}\,,$$
 with $A_1$, $A_2$, $E_{13}\sm A_{13}$ and 
$E_{23}\sm  A_{23}$ non-empty, be an element of $\cF_3$.
Its intersection 
$B_1=A_1\cup E_{12}\cup A_{13}$ with $I_1$ contains $E_{12}$ thus
is not empty. 

We now prove that {\emph when $I$ runs through $\cF_3$ the
sequence $B_1=I\cap I_1$ is totally ordered}. Let 
$I'=A'_1\cup A'_2\cup E_{12}\cup  A'_{13}\cup  A'_{23}$ be another 
element of $\cF_3$. Put $B'_1=I'\cap I_1$, $X=I_1\sm B_1$ and 
$X'=I_1\sm B'_1$. We suppose that $B_1$ and $B'_1$ satisfy no
inclusion. Since they both contain $E_{12}$, Corollary \ref{nocycle}
shows that $I_1=B_1\cup B'_1$ and
that  $X\cap X'=\emptyset$. In particular we obtain
 $E_{13}=A_{13}\cup A'_{13}$ and  $A_2\cap A'_2=\emptyset$. Since 
$A_{13}$ and $A'_{13}$ are distinct from $E_{13}$, the first relation
proves that there are not empty, thus by Lemma \ref{lemAij=Eij}, that
$A_1=A'_1=E_1$. Now we observe that the intersections $A_2\cup
E_{12}\cup A_{23}$ and $A'_2\cup E_{12}\cup A'_{23}$ of $I$ and $I'$
with $I_2$,   satisfy no inclusion, since $A_2$ and $A'_2$ are
non-empty and disjoint. By exchanging $I_1$ and $I_2$ we deduce that 
$A_1$ and $A'_1$, which both coincide with $E_1$, must be disjoint, a
contradiction. 
We conclude (by Corollary \ref{nocycle}) 
that $B_1$ and $B'_1$ satisfy a strict inclusion, for instance
$B'_1\subsetneqq B_1$, i.e.  
$$A'_1\subset A_1\text{\quad and \quad} A'_{13}\subset A_{13}\,,$$
equality holding in at most one inclusion.
In particular, suppose $A'_1=A_1$, then $A_{13}\supsetneqq A'_{13}$ is
not empty, and by the lemma above, $A_1=E_1$. The totally ordered
sequence $(A_1)_{I\in\cF_3}$ contains a strictly increasing sequence 
of non-empty, strict subpaces of $E_1$ (with at most $\abs{E_1}-1$
terms) and of at most  $\abs{E_{13}}$ terms $A_1$ equal to $E_1$
  (associated with a   strictly increasing sequence of strict subspaces
  $A_{13}$ of $E_{13}$).
We have then $\abs{\cF_3 }\le \abs{E_1}+\abs{E_{13}}-1=
p-\abs{E_{12}}-1$ as announced.
\end{proof}

\medskip

Coming back to the proof of Theorem \ref{theocycle},
we conclude that the family $\cF$ contains apart from $I_1$, $I_2$,
$I_3$ at most $3p-3-\abs {W_2}=m-3$ non-exotic terms.
The proof  is complete for $n\le 3p-3$, i.e. $n=m$. 

{\sl Case $n=m+1$, i.e. $n=3p-2$.}
Since $n=m+1$, the $E_{ij}$ are singletons, and their strict subspaces
are empty. The  elements of $\cF_3$ for instance are of the
form $I=A_1\cup A_2\cup E_{12}$, with $A_1$ and $A_2$
non-empty. Actually, we have seen in the above proof that the sequence 
$(A_1)_{I\in\cF_3}$ contains at most $\abs{E_{13}}=1$ term equal to
  $E_1$, so is a strictly totally ordered 
sequence of non-empty subspaces of
  $E_1$, with at most $\abs{E_1}=p-2$ terms. Of course similar remarks 
are valid   for the subspaces $A_2$ of $E_2$.
We  now prove that {\emph if
two families $\cF_i$ are non-empty, one of them at least is a singleton}.
Let $I=A_1\cup A_2\cup E_{12}$ and $I'=A'_1\cup A'_3\cup E_{13}$ be
elements of $\cF_3$ and $\cF_2$ respectively.
First, $A_1$ and $A'_1$ must satisfy an inclusion.
Otherwise, $(I,I_1,I',I_3,I_2)$ should be a cycle for the relation
$\sim$. Indeed, the sets of indices of weight $2$ in $I\cap I_1$,
$I'\cap I_1$, $I'\cap I_3$, $I_3\cap I_2$ and $I\cap I_2$
(respectively $A_1\sm A'_1$, $A'_1\sm A_1$, $A'_3\sm A_3=A'_3$,
$E_{23}\sm (A_{23}\cup A'_{23})=E_{23}$ and $A_2\sm A'_2=A_2$),
should be all non-empty. Now, since
$\abs{I\cup I'\bigcup_j I_j}=\abs{\bigcup_j I_j}=m=n-1$, by
 \ref{lemW2}, the subsets above should all be singletons, 
in particular $A_2$ and $A'_3$, implying  that $A_1$ and $A'_1$ should
contains $p-2$ elements, i.e. both coincide with $E_1$, a
contradiction. 
Therefore we may suppose $A'_1\subset A_1$. The
intersections $B_1=A_1\cup E_{12}$ and  $B'_1=A'_1\cup E_{13}$ of $I$
and $I'$ with $I_1$ satisfy $B_1\cap B'_1=A'_1\ne\emptyset$, and thus,
by Corollary \ref{nocycle}, $E_1=A_1\cup A'_1$, i.e. $A_1=E_1$,
which specifies  uniquely $I=E_1\cup \{a_2\}\cup E_{12}$  
in $\cF_3$: $\abs{\cF_3}=1$ as announced.
 
If $\cF_1$ is empty, we have $\abs{\bigcup \cF_i}\le 1+(p-2)$,
and $\cF$ contains at most $p+2$ non-exotic elements.

Otherwise, let $I"=A"_2\cup A"_3\cup E_{23}$
be an element of $\cF_1$. By exchanging $I"$ with $I'$ or $I$,
we know that $A"_2$ and $A_2$  on the one hand, $A"_3$ and $A'_3$
on the other hand, must satisfy an inclusion, and that
the larger of the subsets coincides with $E_2$ or $E_3$ respectively.
We thus have
$$\abs{A_2}=1\Rightarrow A"_2=E_2\Rightarrow \abs{A"_3}=1
\Rightarrow A_3=E_3\,.$$: 
$\bigcup \cF_i$ contains at most $3$ elements, of the form
$E_1\cup \{a_2\}\cup E_{12}$, $\{a'_1\}\cup E_3\cup E_{13}$ and
$E_2\cup \{a"_3\}\cup E_{23}$, with uniquely determined elements
$a_i, a'_i$ or $a"_i$ in $E_i$.
 
We conclude that $\abs{\bigcup \cF_i}\le \max(3,p-1)$, and thus
that in the case $n=3p-2$, $\cF$ contains $p+2$ (resp. $6$) non-exotic
elements if $p\ge 4$ (resp. $p=3$). 

To complete the proof of the theorem, it remains to discuss the
occurrence of the ``exotic'' elements when 
$(p,n)=(3,7)$ or $(4,10)$. Actually, in both cases, an exotic element,
say $I$, described by Lemma \ref{lemnoAij} is inconsistent with an
element, say $J$, of $\bigcup \cF_j$: there exists $k=1,2,3$ such that
the set $\{I,J,I_k\}$ contradicts  Proposition \ref{proptrio}.
So $\cF$ contains at most $6+1=n$ (resp. $6=p+2$) elements when 
$p=3$ (resp. $p=4$). This completes the proof of Theorem \ref{theocycle}.

\section{Kissing number of a lattice of index $2$ , maximal length }
\medskip
The goal of this section is to prove Theorem \ref{noperf}
by giving an explicit upper bound for the 
number $s$ of pairs $\pm x$ of minimal vectors of the lattice, bound
depending on the dimension $n$ modulo $6$.

\begin{theorem} \label{theokissing}
Let $L$ be a lattice of dimension $n\ge 6$, index $2$ with length $\ell=n$.
Bounds for the half kissing number $s$ of $L$ are given in the
following table.
$$\vbox{\offinterlineskip\halign{#&#&#&#&#&#&#&#&#&#&#&#&#&#&#\cr
%%\noalign{\hrule}
&\cc{n \mod 6}&\tvi&\cc{\text{upper bound for }s}\cr
\noalign{\hrule}
%&\cc{}&\tvii&\cc{}\cr
\noalign{\hrule}
&\cc{0} &\tvj&\cc{ 19\ \text{ if }n=6\quad ;  (2n^2+24n-45)/9\ \text{ if }n\ge 12}\cr
%&\cc{}&\tvii&\cc{}&\tvii&\cc{}\cr
\noalign{\hrule}
&\cc{1} &\tvj&\cc{24\ \text{ if }n=7\quad ; (2n^2+20n-13)/9 \ \text{ if }n\ge 13}\cr
\noalign{\hrule}
&\cc{2} &\tvj&\cc{ 32\ \text{ if }n=8\quad ;  (2n^2+22n-25)/9\ \text{ if }n\ge 14}\cr
\noalign{\hrule}
&\cc{3} &\tvj&\cc{ 37\ \text{ if }n=9\quad  ; (2n^2+24n-54)/9\ \text{ if }n\ge 15}\cr
\noalign{\hrule}
&\cc{4\ }&\tvj&\cc{ 44\ \text{ if }n=10\quad ; (2n^2+20n-4)/9\  \text{ if }n\ge 16}\cr
\noalign{\hrule}
&\cc{5} &\tvj&\cc{ (2n^2+22 n -34)/9 \quad ( n\ge 11)}\cr
}}
$$
\end{theorem}
\begin{proof}
 To compute the number $s=\frac{\abs{S(L)}}2$ 
of pairs of minimal vectors of the lattice
 $L$, we use the following description 
$$S(L)=S(L_0)\cup \S_0\cup S_1\cup S_2\cup\dots \cup S_{\lfloor \frac
  n2\rfloor}\,,$$
where $S(L_0)$ stands for the set of minimal vectors of the lattice
$L_0=\langle e_1,e_2,\dots, e_n\rangle$,  and 
$S_p$ for the set of pairs $\pm x$ where $x$ is a minimal vectors of type $p$.
%%i.e. of the form $\pm(e-\sum_{i\in I}e_i)$ with $\abs{I}=p$.
Let $t_p=\frac{\abs{S_p}}2$ denote the number of such pairs. 
Since $S(L_0)=\{\pm e_1,\pm e_2,\dots \pm e_n\}$ and $S_0=\{\pm e\}$,
we obtain
$$s=n+1+\sum_{p=1}^{\lfloor\frac n2\rfloor} t_p\,$$
where we shall use of the estimations of the $t_p$ given in the sections
above, and the sharper one obtained for $t_1+t_2$ in \ref{cort12}:
$$s\le n+10+\sum_{p=3}^{\lfloor\frac n2\rfloor} t_p\,.$$
For $p\ge 3$, let $T_1$ and $T_2$  denote the bounds for $t_p$ given
by Theorems \ref{theoI1I2} and \ref{theocycle}.
If $S_p$  may contain a cycle of length $3$, i.e.,by \ref{theocycle}, 
if $\frac{n+2}3\le p\le\frac{n-1}2$, 
 we obtain for
$t_p$ the estimation $t_p\le \max(T_1,T_2)$. Otherwise, $t_p\le T_1$.
 Now suppose $p=\frac{n+2}3$. If $(p,n)=(3,7)$, $T_1$  and $T_2$
coincide with $n=7$ ; if $p\ge 4$, $T_2=p+2<p+5\le T_1$, and again
$\max (T_1,T_2)=T_1$. The bound $T_2$
is to take into account for  the integers $p$ such that
$\frac{n+2}3< p\le\frac{n-1}2$, i.e. for the elements of
$$\cP=\{p_1,p_1+1,\dots,p_k\},\ \text{ with }\
p_1=\lceil\frac{n+3}3\rceil,\ p_k=\lfloor\frac{n-1}2\rfloor\,.$$
Actually, $\cP$ is empty for $n=6$, $n=8$ and $n=10$, it contains 
the type $p=3$ for $n=7$ only (and then $T_1=T_2=7$), the type $p=4$ for
$n=9$ (and then $T_1=T_2=9$).
Thus, for $p\in \cP$ and $n\le 10$, we have $t_p\le T_1=\max(T_1,T_2)$,
while for $p\in\cP$ and $n\ge 11$, we have $t_p\le T_2=\max(T_1,T_2)$
(since then $T_2=n\ge \frac{n-1}2+6\ge p+6\ge T_1$).

For $n\le 10$, we  sum up the bounds  given by \ref{theoI1I2}
\ref{4,8} and \ref{cort12}:

$n=6$: $s\le 7+9+4=20$ (bound to be improved below);

$n=7$: $s\le 8+9+7=24$;

$n=8$: $s\le 9+9+8+6=32$;

$n=9$: $s\le 10+9+8+9=36$;

$n=10$: $s\le 11+9+8+10+6=44$.

From now on, we suppose $n\ge 11$.

Let $\Sigma_1$ denote the sum of the bounds $T_1$ given by
\ref{theoI1I2} for $t_p,\ 3\le p\le
\lfloor\frac n2\rfloor$:
$\Sigma_1=8+\sum_{p=4}^{\lfloor\frac n2\rfloor}
(p+6)-\va=\lfloor\frac n2\rfloor(\frac{\lfloor\frac n2\rfloor
  +13}2)-16-\va\,,$
with $\va=1$ if $n$ is odd, $\va=5$ if $n$ is even:
$$\Sigma_1=\frac{n^2+24n-161}8 \text{ if $n$ is odd},\quad
\Sigma_1=\frac{n^2+26n-168}8   \text{ if $n$ is even }
$$
%%$$\begin{aligned}\Sigma_1=&\frac{n^2+24n-161}8 \text{ if $n$ is odd}\\
%%                     = &\frac{n^2+26n-168}8   \text{ if $n$ is even}.
%%\end{aligned}
%%$$
For $p\in\cP$ and $n\ge 11$ (thus $p\ge 5$), we must replace
the bound $T_1$ by the bound $T_2=n$. Let $\Sigma_2$ denote the
sum of these correcting terms $T_2-T_1=n-(p+6)$ (resp. $n-(p+5)$
for $p\ne \frac{n-1}2$ (resp. $=\frac{n-1}2$). We have
$$\begin{aligned}
\Sigma_2=&\sum_{p\in\cP} (n-6-p)+\va, \text{ with }\va=1
\text{ if $n$ is odd}\\
=&k(n-p_1-6)-1-2-\cdots-(k-1)+\va, \text{ with }k=(p_k-p_1+1)\\
=&(p_k-p_1+1)(n-\frac{p_1+p_k}2-6)+\va\,.
\end{aligned}
$$

One  easily checks the following expressions of this correcting term,
depending  on $n$ modulo $6$. 

\hbox{}\hskip3.3cm
\vbox{
\begin{tabular}{c|c}
$n$&$72 \Sigma_2$\\
\hline
$\equiv 0$&$7n^2-114n+432$\\

$\equiv 1$&$7n^2-128n+625$\\

$\equiv 2$&$7n^2-130n+592$\\

$\equiv 3$&$7n^2-96n+297$\\

$\equiv 4$&$7n^2-146n+76$\\

$\equiv5$&$7n^2-112n+457$

\end{tabular}}

We now use the inequality
$s\le n+1+9+\Sigma_1 +\Sigma_2$
to obtain the table of Theorem \ref{theokissing} for $n\ge 7$.
Of course the bounds for $s=n+\sum t_p$ obtained by bounding separately the
$t_p$ are not optimal. This is the case when $(p,n)=(3,6)$: the maximal value 
$4$ of $t_3$  is inconsistent with the
maximal value $9$ of $t_1+t_2$,  which leads to the bound 
$s\le 6+1+9+3=19$ instead of $20$.

[{\small Suppose $t_3=4$. The four sets
$I$ of type $3$ are, up to permutation, 
$I_1=\{1,2,3\}$, $I_2=\{1,3,4\}$, $I_3=\{1,4,5\}$,  $I_4=\{1,5,2\}$;
then, no index $i\in\{1,2,\dots,6\}$ has valency $3$ for the relation
``$i\equiv j$ if $e-e_i-e_j$ is minimal'', as we now prove. There are
$3$ cases to consider according as the weight of $i$ (with respect to 
the $I_j$) is equal to  $4$, $2$ or $0$.
First, $i=1$ has at most valency $2$: Proposition \ref{lemW1}
prevents  $1\equiv 6$, and proves that $1\equiv 2$ is inconsistent with
$1\equiv 3$ or $1\equiv 5$. For  $i=2$, Proposition \ref{lemW1} proves
that $2\equiv 3$ is inconsistent with $2\equiv 6$, and by Proposition
\ref{lemW2} we see that $2\equiv 5$ is not possible, and that $2\equiv 3$ is
 inconsistent with $2\equiv 1$ (for instance the sets $\{2,1\}$,
 $\{2,3\}$, 
and  $I_2=\{1,3,4\}$ form a cycle of length $3$ and $m=4$ indices,
impossible for $n=6$). The same argument implies that $6\equiv 2$ is
inconsistent with $6\equiv 3$ (and $6\equiv 5$).
Thus, by Corollary \ref{cort12}, we have $t_1+t_2\le 8$. }]
\end{proof}
The difference between  $\frac{n(n+1)}2$ and the bound  for $s$
 given in \ref{theokissing}  takes the values
$$2,4,4,9,11,16,19,26,30,36,44,\dots$$
for  $n=6,7,8,9,10,11,12,13,14,15,16,\dots$,
is always positive and monotone increasing, and asymptotic to
$5n^2/18$ as $n\to\infty$. This completes the proof of  Theorem \ref{noperf}.

\end{document}